\documentclass[leqno,11pt]{scrartcl}
\usepackage[utf8]{inputenc}

\usepackage{amsmath,amssymb,amsthm}
\usepackage{amscd}
\usepackage{setspace} 
\usepackage{enumerate}
\usepackage{array}
\usepackage{tabularx}
\usepackage{multirow}
\usepackage{booktabs}

\theoremstyle{definition}
\newtheorem{definition}{Definition}

\newtheorem*{remarkson}{Remarks}
\theoremstyle{plain}
\newtheorem{theorem}{Theorem}

\newtheorem{lemma}[definition]{Lemma}
\newtheorem{corollary}{Corollary}

\theoremstyle{remark}

\newcommand{\C}{\mathbb{C}}
\newcommand{\F}{\mathbb{F}}

\renewcommand{\H}{\mathbb{H}}
\newcommand{\K}{\mathbb{K}}
\newcommand{\N}{\mathbb{N}}

\newcommand{\Q}{\mathbb{Q}}
\newcommand{\R}{\mathbb{R}}
\newcommand{\Z}{\mathbb{Z}}

\newcommand{\Dcal}{\mathcal{D}}

\newcommand{\Hcal}{\mathcal{H}}
\newcommand{\Ical}{\mathcal{I}}

\newcommand{\Ocal}{\mathcal{O}}

\newcommand{\trace}{\operatorname{trace}\,}
\newcommand{\diag}{\operatorname{diag}\,}
\newcommand{\her}{\operatorname{Her}}
\renewcommand{\Re}{\operatorname{Re}\,}
\newcommand{\sym}{\operatorname{Sym}}

\newcommand{\oh}{{\scriptstyle{{\cal O}}}}

\begin{document}

\begin{center}
\begin{huge}
\begin{spacing}{1.0}
\textbf{Congruence Subgroups and Orthogonal Groups}  
\end{spacing}
\end{huge}

\bigskip
by
\bigskip

\begin{large}
\textbf{Adrian Hauffe-Waschbüsch\footnote{Adrian Hauffe-Waschbüsch, Lehrstuhl A für Mathematik, RWTH Aachen University, D-52056 Aachen, adrian.hauffe@matha.rwth-aachen.de}} and
\textbf{Aloys Krieg\footnote{Aloys Krieg, Lehrstuhl A für Mathematik, RWTH Aachen University, D-52056 Aachen, krieg@rwth-aachen.de}}
\end{large}
\vspace{0.5cm}\\
January 2021
\vspace{1cm}
\end{center}
\begin{abstract}
We derive explicit isomorphisms between certain congruence subgroups of the Siegel modular group, the Hermitian modular group over an arbitrary imaginary-quadratic number field and the modular group over the Hurwitz quaternions of degree $2$ and the discriminant kernels of special orthogonal groups $SO_0(2,n)$, $n= 3,4,6$. The proof is based on an application of linear algebra adapted to the number theoretical needs.
\end{abstract}
\noindent\textbf{Classification: 11F46, 11F55}
\vspace{2ex}\\

\newpage
\section{Introduction}

In the classical theory modular forms such as theta series are described in the setting of $Sp_n(\R)$ (cf. \cite{F}). About 20 years ago Borcherds (\cite{Bo}) established a product expansion for modular forms on $SO(2,n)$.

There are well-known isomorphisms between the projective symplectic group $PSp_2(\R)$, the projective special split unitary group $PSU(2,2;\C)$ and the projective quaternionic symplectic group $PSp_2(\H)$ on the one hand and the connected component of the identity of the projective special orthogonal group $PSO(2,n;\R)$, $n\in \{3,4,6\}$, on the other hand. 
In the first two cases an explicit form based on linear algebra can be found in \cite{GKver} and \cite{KRaW}. If one is interested in modular forms, it is necessary to adapt this isomorphism to number theoretical needs in order to identify modular forms with respect to discriminant kernels (cf. \cite{WW1}, \cite{WW2}, \cite{Will}) with Siegel and Hermitian modular forms with respect to congruence subgroups. This approach can be extended to the quaternions, where the lack of a classical notion of determinants leads to particular difficulties.

Given a non-degenerate symmetric even matrix $T\in\Z^{m\times m}$ let 
\[
 SO(T;\R):= \{U\in SL_m(\R);\;U^{tr} TU = T\}
\]
denote the attached special orthogonal group. Let $SO_0(T;\R)$ stand for the connected component of the identity matrix $I$ and $SO_0(T;\Z)$ for the subgroup of integral matrices. The \emph{discriminant kernel}
\[
 \Dcal(T;\Z):=\{U\in SO_0(T;\Z); \; U\in I+\Z^{m\times m}T\}
\]
is a normal subgroup of $SO_0(T;\Z)$. Given $N\in \N$ we moreover define
\[
 U(N) \oplus T: = \begin{pmatrix}
                   0 & 0 & N \\ 0 & T & 0 \\ N & 0 & 0
                  \end{pmatrix}
\]
for the orthogonal sum with the rescaled hyperbolic plane.

\section{Siegel modular group}

Denote by 
\[
\Gamma_2(\Z):= \{M\in \Z^{4\times 4};\;M^{tr} JM=J\},\;\; J= J^{(4)} = \begin{pmatrix}
                                                              0 & I \\ -I & 0
                                                             \end{pmatrix},
                                                             \;\; I= I^{(2)} = \begin{pmatrix}
                                                              1 & 0 \\ 0 & 1
                                                             \end{pmatrix},
\]
the \emph{Siegel modular group} of degree $2$. Throughout the paper we will always choose a block decomposition of $M$ into $2\times 2$ blocks
\begin{gather*}\tag{1}\label{gl_1}
 M= \begin{pmatrix}
     A & B \\ C & D
    \end{pmatrix},\;
 A= (a_{ij}),\; B=(b_{ij}), \; C= (c_{ij}), D=(d_{ij}).
\end{gather*}
Now let 
\begin{gather*}\tag{2}\label{gl_2}
 S_0 = U(1)\oplus(-2) = \begin{pmatrix}
                         0 & 0 & 1 \\ 0 & -2 & 0 \\ 1 & 0 & 0
                        \end{pmatrix},\;\;
 S_1 = U(1)\oplus S_0 = \begin{pmatrix}
                         0 & 0 & 1 \\ 0 & S_0 & 0 \\ 1 & 0 & 0
                        \end{pmatrix}.
\end{gather*}
For $2\times 2$ matrices the adjoint is defined by
\begin{gather*}\tag{3}\label{gl_3n}
 \begin{pmatrix}
  \alpha & \beta \\ \gamma & \delta
 \end{pmatrix}^{\sharp} = \begin{pmatrix}
			  \delta & -\beta \\ -\gamma & \alpha
			  \end{pmatrix}.
\end{gather*}			  
Denote the symmetric $2\times 2$ matrices by $\sym_2(\R)$ and define
\begin{gather*}\tag{4}\label{gl_4n}
\varphi:\sym_2(\R) \to \R^3, \;\; \begin{pmatrix}
                                   \alpha & \beta \\ \beta & \gamma
                                  \end{pmatrix}
 \mapsto \begin{pmatrix}
          \alpha \\ \beta \\ \gamma
         \end{pmatrix}.
\end{gather*}
We fix a notation for $\widetilde{M}$
\begin{gather*}\tag{5}\label{gl_5n}
 \widetilde{M} = \begin{pmatrix}
                         \alpha & a^{tr}S_0 & \beta \\ b & K & c \\ \gamma & d^{tr} S_0 & \delta
                        \end{pmatrix},\;
 \alpha, \beta, \gamma, \delta \in \R.
\end{gather*}
Given $M\in Sp_2(\R)$ of the form \eqref{gl_1} observe that
\begin{align*}
M\langle Z\rangle & := (AZ+B) (CZ+D)^{-1} = \tfrac{1}{\det(CZ+D)} \bigl(\det Z\cdot AC^\sharp +AZD^\sharp + BZ^\sharp C^\sharp + BD^\sharp\bigr), \\
\det M\{Z\} & := \det (CZ+D) = \det Z \cdot\det C + \trace (Z^\sharp C^\sharp D) + \det D,
\end{align*}
and define $\widetilde{M}\in \R^{5\times 5}$ in \eqref{gl_5n} by
\begin{gather*}\tag{6}\label{gl_6}
 \alpha = \det A, \; \beta = -\det B, \; \gamma = - \det C, \; \delta = \det D,
\end{gather*}
\begin{gather*}\tag{7}\label{gl_7}
 a = -\varphi(A^{\sharp}B),\; b = -\varphi(AC^{\sharp}), \; c = \varphi(BD^{\sharp}), \; d = \varphi(C^{\sharp} D).
\end{gather*}
\begin{gather*}\tag{8}\label{gl_8}
\begin{split}
 & \text{$K$ is the matrix, which represents the endomorphism}\\
 & f_M:\sym_2(\R) \to \sym_2(\R),\;\; Z\mapsto AZD^{\sharp}+BZ^{\sharp}C^{\sharp},  \\
 & \text{with respect to the basis $\left(\begin{smallmatrix}
                                           1 & 0 \\ 0 & 0
                                          \end{smallmatrix}\right), \;
 \left(\begin{smallmatrix}
  0 & 1 \\ 1 & 0
 \end{smallmatrix}\right), \;
 \left(\begin{smallmatrix}
  0 & 0 \\ 0 & 1
 \end{smallmatrix}\right).$}
 \end{split}
\end{gather*}
Our first result is
\begin{theorem}\label{theorem_1} 
 a) $\Gamma_2(\Z)/\{\pm I\}$ is isomorphic to $\Dcal(S_1;\Z)$ via the map $\pm M\mapsto \widetilde{M}$ given by \eqref{gl_6}, \eqref{gl_7}, \eqref{gl_8}. \\
 b) Given $N\in\N$ the discriminant kernel $\Dcal(NS_1;\Z)$ is isomorphic to the principal congruence subgroup of level $N$
 \[
  \{M\in\Gamma_2(\Z);\; M\equiv \varepsilon I \bmod{N}, \,\varepsilon \in\Z,\,\varepsilon^2\equiv 1\bmod{N}\}/\{\pm I\}
 \]
under the map in a).  \\
c) Given $n,N\in\N$, $n\mid N$ the map in a) yields an isomorphism between the congruence subgroup
\begin{gather*}\tag{9}\label{gl_9}
\Biggl\{ M\in\Gamma_2(\Z), \;\,
\begin{split}
& \; a_{11} \equiv a_{22}\equiv d_{11}\equiv d_{22} \equiv \varepsilon \bmod{n},\, \varepsilon\in\Z,\,\varepsilon^2\equiv 1\bmod{n}, \\
& \det A\equiv \det D\equiv 1\bmod{N},\, a_{21}\equiv b_{22}\equiv d_{12} \equiv 0\bmod{n}, \\
& \; c_{11}\equiv 0\bmod{Nn},\,c_{12}\equiv c_{21} \equiv c_{22}\equiv 0\bmod{N}
\end{split}\;\Biggr\} \Big/ \{\pm I\}
\end{gather*}
and
\begin{gather*}\tag{10}\label{gl_10}
 F\;\Dcal\bigl(U(N)\oplus U(n) \oplus (-2);\Z\bigr)\; F^{-1}\subseteq \Dcal(S_1;\Z), \, F=\diag(1,1,1,n,N).
\end{gather*}
\end{theorem}

\begin{proof}
 a) This is the case $N=1$ of Lemma 5 in \cite{GKver}. \\
 b) At first \eqref{gl_6}, \eqref{gl_7}, \eqref{gl_8} show that the image of the principal congruence subgroup is contained in the discriminant kernel. Given $\widetilde{M} \in \Dcal(NS_1;\Z)$ of the form \eqref{gl_5n} part a) yields $B\equiv C \equiv 0\bmod{N}$, because $\det A$ and $\det D$ are coprime to $N$ in \eqref{gl_7}. Now observe (cf. \cite{K3}, p. 44) that 
 \[
  \pm M_0 = \pm(I\times J) \mapsto \widetilde{M}_0 = \begin{pmatrix}
                                                      -J & 0 & 0 \\ 0 & 1 & 0 \\ 0 & 0 & J
                                                     \end{pmatrix}, \;\; J=J^{(2)}.
 \]
If we conjugate by $M_0$ resp. $\widetilde{M}_0$, the same argument shows that $M$ is congruent to a diagonal matrix $\bmod{N}$. The determinantal conditions on the blocks of $M$ and $M_0MM^{-1}_0$ as well as $AD^{tr} \equiv I\bmod{N}$ imply
\[
 M\equiv \varepsilon I\bmod{N}, \; \varepsilon\in\Z,\;\; \varepsilon^2\equiv 1\bmod{N}.
\]
c) The image of the congruence subgroup \eqref{gl_9} is contained in \eqref{gl_10} by \eqref{gl_6}, \eqref{gl_7}, \eqref{gl_8}. Given $\widetilde{M} \in F\;\Dcal \bigl(U(N)\oplus U(n) \oplus (-2);\Z\bigr) \;F^{-1}$ we have
\[
 \alpha \equiv \delta \equiv 1\bmod{N}, \;\; b,d\equiv 0\bmod{N}.
\]
We obtain $C\equiv 0\bmod{N}$. Now we apply the same procedure to $\widetilde{M}_0\widetilde{M}\widetilde{M}^{-1}_0$. This leads to 
\[
 a_{21}\equiv b_{22} \equiv d_{12} \equiv 0\bmod{n}.
\]
The determinantal conditions on the $A$- and $D$-blocks as well as $AD^{tr}\equiv I \bmod{n}$ imply that all the diagonal entries of $M$ are congruent to $\varepsilon\bmod{n}$ for some $\varepsilon\in\Z$ satisfying $\varepsilon^2\equiv 1\bmod{n}$. From these congruences and the fact that the last entry of $b=-\varphi(AC^{\sharp})$ is divisible by $Nn$, we get $c_{11}\equiv 0\bmod{Nn}$.
\end{proof}

\begin{remarkson}
 a) The projective symplectic group in the underlying isomorphism of real Lie groups entails the appearance of the projective principal congruence subgroup in Theorem \ref{theorem_1}, which differs from the common one in the non-projective setting (cf. \cite{Kl2}).  \\
 b) It follows from c) that $\Dcal\bigl(U(N)\oplus U(1) \oplus (-2);\Z\bigr)$ is isomorphic to the more familiar congruence subgroup
 \[
  \{M\in \Gamma_2(\Z);\; C\equiv 0\bmod{N},\, \det A\equiv \det D \equiv 1\bmod{N}\} \big/\{\pm I\}.
 \]
 c) In the language of lattices c) of Theorem \ref{theorem_1} refers to the orthogonal group with respect to $U(N) \oplus U(n) \oplus A_1(-1)$. 
 The case of an arbitrary orthogonal sum with two hyperbolic planes over $\Z$ can be reduced to Theorem \ref{theorem_1} c) by the elementary divisor theorem.
\end{remarkson}

\section{Hermitian modular group}

Let $\K= \Q(\sqrt{-m})$, $m\in\N$ squarefree, be an imaginary quadratic number field with discriminant $d_\K$ and ring of integers $\oh_\K =  \Z+\Z \omega_\K$, given by 
\begin{gather*}\tag{11}\label{gl_11}
 d_\K =
 \begin{cases}
  -m,& \\ -4m,&
 \end{cases}
 \;\; \omega_\K = 
 \begin{cases}
  (m+\sqrt{-m})/2, & \text{if}\; m\equiv 3\bmod{4}, \\ m+\sqrt{-m}, & \text{else}.
 \end{cases}
\end{gather*}
The \emph{Hermitian modular group} of degree $2$ with respect to $\K$ (cf. \cite{B2}) is defined by 
\[
 \Gamma_2(\oh_\K):=\bigl\{M\in SL_4(\oh_\K);\; \overline{M}^{tr} JM=J\bigr\}.
\]
In this case we define 
\begin{gather*}\tag{12}\label{gl_12}
 S_\K = \begin{pmatrix}
         2 & 2\Re(\omega_\K) \\ 2 \Re(\omega_\K) & 2 |\omega_\K|^2
        \end{pmatrix}, \;\; S_0 = 
        \begin{pmatrix}
         0 & 0 & 1 \\ 0 & -S & 0 \\ 1 & 0 & 0
        \end{pmatrix}, \;\; S_1 = 
        \begin{pmatrix}
         0 & 0 & 1 \\ 0 & S_0 & 0 \\ 1 & 0 & 0
        \end{pmatrix} \in \Z^{6\times 6}.
\end{gather*}
Note that
\begin{gather*}\tag{13}\label{gl_13}
 (\Z\times \Z) S_\K = 
 \begin{cases}
  \Z\times m\Z, & \text{if}\; d_\K \;\text{is odd}, \\
  2\Z\times 2m\Z, & \text{if}\; d_\K \;\text{is even}.
 \end{cases}
\end{gather*}
In this setting we consider the Hermitian $2\times 2$ matrices $\her_2(\C)$ instead of $\sym_2(\R)$ in section 2 and replace \eqref{gl_4n} by
\begin{gather*}\tag{4'}\label{gl_4strich_n}
 \varphi:\her_2(\C) \to \R^4,\;\;\begin{pmatrix}
                                 \alpha & \beta+\gamma\omega_\K \\ \beta +\gamma\overline{\omega}_\K & \delta
                                \end{pmatrix}
                                \mapsto (\alpha,\beta,\gamma,\delta)^{tr}.
\end{gather*}
We adopt the notions of section 2. Given $M\in SU(2,2;\C)$ with block decompositions as in \eqref{gl_2}, we consider $\widetilde{M}\in SO(S_1;\R)$ with $S_1$ from \eqref{gl_12} in the notion of \eqref{gl_5n} using \eqref{gl_6}, \eqref{gl_7} with $\varphi$ from \eqref{gl_4strich_n} as well as
\begin{gather*}\tag{8'}\label{gl_8strich}
 \begin{split}
  & K\; \text{is the matrix, which represents the endomorphism} \\
  & f_M:\her_2(\C)\to \her_2(\C), \;\; Z\mapsto AZD^{\sharp}+BZ^{\sharp}C^{\sharp},  \\
  & \text{with respect to the basis}\; 
  \left(\begin{smallmatrix}
         1 & 0 \\ 0 & 0
        \end{smallmatrix}\right), \;
  \left(\begin{smallmatrix}
         0 & 1 \\ 1 & 0
        \end{smallmatrix}\right),\;
  \left(\begin{smallmatrix}
         0 & \omega_\K \\ \overline{\omega}_\K & 0
        \end{smallmatrix}\right),\;
  \left(\begin{smallmatrix}
         0 & 0 \\ 0 & 1
        \end{smallmatrix}\right).
 \end{split}
\end{gather*}
If we consider congruences $\bmod{\,N}$, i.e. $\bmod{\,N\oh_\K}$, $N\in\N$, the result is
\begin{theorem}\label{theorem_2} 
 a) $\Gamma_2(\oh_\K)/\{\pm I\}$ is isomorphic to $\Dcal(S_1;\Z)$ the map $\pm M\mapsto \widetilde{M}$ given by \eqref{gl_4strich_n}, \eqref{gl_6}, \eqref{gl_7}, \eqref{gl_8strich}.  \\
 b) Given $N\in\N$ the discriminant kernel $\Dcal(NS_1;\Z)$ is isomorphic to the principal congruence subgroup of level $N\oh_\K$,
\[
 \{M\in\Gamma_2(\oh_\K);\;M\equiv \varepsilon I\bmod{N},\,\varepsilon\in\Z,\,\varepsilon^2\equiv 1\bmod{N}\}/\{\pm I\}
\]
under the map in a). \\
c) Given $n,N\in\N$, $n\mid N$ the map in a) yields an isomorphism between the congruence subgroup
\begin{gather*}\tag{14}\label{gl_14}
\Biggl\{ M\in\Gamma_2(\oh_\K); \;\,
\begin{split}
& \; a_{11} \equiv d_{11}\equiv \overline{a}_{22} \equiv \overline{d}_{22} \equiv \varepsilon \bmod{n},\, \varepsilon\in\oh_\K,\,\varepsilon \overline{\varepsilon}\equiv 1\bmod{n}, \\
& \det A\equiv \det D\equiv 1\bmod{N},\, a_{21}\equiv b_{22}\equiv d_{12} \equiv 0\bmod{n}, \\
& \; c_{11}\equiv 0\bmod{Nn},\,c_{12}\equiv c_{21} \equiv c_{22}\equiv 0\bmod{N}
\end{split}\;\Biggr\} \Big/\{\pm I\}
\end{gather*}
and
\[
F\;\Dcal\bigl(U(N)\oplus U(n) \oplus(-S_\K);\Z\bigr)\; F^{-1}\subseteq \Dcal(S_1;\Z),\;\; F= \diag(1,1,1,1,n,N). 
\]
\end{theorem}

\begin{proof}
 a) This is Theorem 3 in \cite{KRaW}.  \\
 c) Now let $N>1$. Proceed in the same way as in the proof of Theorem \ref{theorem_1}\,c) and observe that $A\overline{D}^{tr}\equiv I\bmod{n}$, where in this situation 
 \[
  \pm M_0 = \pm(I\times J)\mapsto \widetilde{M}_0 = \begin{pmatrix}
                                                     -J & 0 & 0  \\  0 & I & 0  \\  0 & 0 & J
                                                    \end{pmatrix}, \;\;
  I = I^{(2)}, \;\; J=J^{(2)}.
 \]
 b) Proceed as before. Note that any $M$ in the preimage satisfies
 \[
  M\equiv \diag(\varepsilon,\overline{\varepsilon},\varepsilon,\overline{\varepsilon})\bmod{N},\; \varepsilon\in\oh_\K, \;\varepsilon,\overline{\varepsilon}\equiv 1\bmod{N}.
 \]
As $\widetilde{M}$ belongs to the discriminant kernel, we conclude
\[
 f_M \left(\begin{pmatrix}
            0 & 1 \\ 1 & 0
           \end{pmatrix}\right) \equiv \begin{pmatrix}
					\ast & \varepsilon^2 \\ \overline{\varepsilon}^2 & \ast
					\end{pmatrix}
					\bmod{N}\;\text{resp.}\;\bmod{2N}
\]
from \eqref{gl_13}, hence
\[
 \varepsilon^2 \equiv 1, \; \varepsilon \equiv \overline{\varepsilon} 
 \begin{cases}
  \bmod{\,N}, & \text{if $d_\K$ is odd}, \\
  \bmod{\,2N}, & \text{if $d_\K$ is even}.
 \end{cases}
\]
Thus $\varepsilon$ may be chosen in $\Z$.
\end{proof}

In order to obtain congruence subgroups, which contain $Sp_2(\Z)$ properly, we need an ideal $\Ical$ in $\oh_\K$ satisfying $\Ical=\overline{\Ical}$. Thus we consider a squarefree divisor $N\mid d_\K$ and the integral ideal $\Ical_N$ of reduced norm $N$. As $N\mid |\omega_\K|^2$ in \eqref{gl_11}, we get
\begin{gather*}\tag{15}\label{gl_15}
\Ical_N = \Z N+\Z\omega_\K,\;\; \Ical^2_N = N\oh_\K. 
\end{gather*}
In this case we define
\begin{gather*}\tag{16}\label{gl_16}
 T = \begin{pmatrix}
      2N & 2\Re(\omega_\K) \\ 2\Re(\omega_\K) & 2|\omega_\K|^2/N
     \end{pmatrix}.
\end{gather*}
We observe that
\begin{gather*}\tag{17}\label{gl_17}
 (\Z\times \Z) T\begin{pmatrix}
                 1 & 0 \\ 0 & N
                \end{pmatrix}= (N\Z\times \Z)S_\K.
\end{gather*}

\begin{theorem}\label{theorem_3} 
Given a squarefree divisor $N$ of $d_\K$ the map $\pm M\mapsto \widetilde{M}$ in Theorem \ref{theorem_2} yields an isomorphism between the principal congruence subgroup $\bmod{\;\Ical_N}$
\begin{gather*}\tag{18}\label{gl_18}
 \{M\in\Gamma_2(\oh_\K);\; M\equiv \varepsilon I\bmod{\,\Ical_N},\, \varepsilon\in\Z,\,\varepsilon^2\equiv 1\bmod{N}\}/\{\pm I\}
\end{gather*}
and
\begin{gather*}\tag{19}\label{gl_19}
 G^{-1}_N \Dcal\bigl(U(N)\oplus U(N) \oplus (-T);\Z\bigr)G_N,\; G_N= \diag(1,1,1,N,1,1).
\end{gather*}
\end{theorem}

\begin{proof}
 Note that $H\in \Ical^{2\times 2}_N$, $H=\overline{H}^{tr}$ satisfies
 \[
  \varphi(H) \equiv (0,0,\ast,0)^{tr}\bmod{N}
 \]
because of \eqref{gl_15}. Using \eqref{gl_4strich_n} - \eqref{gl_8strich} a verification shows that the images of the matrices in \eqref{gl_18} are contained in \eqref{gl_19}. Any representative of $\oh_\K/\Ical_N$ may be chosen in $\Z$. Hence we conclude in the same way as in the proof of Theorem \ref{theorem_2} that any $M$ in the preimage of 
\[
 G^{-1}_N \;\Dcal\bigl(U(N) \oplus U(N) \oplus(-T);\Z\bigr)\; G_N \cap \Dcal(S_1;\Z)
\]
satisfies
\[
 M\equiv \varepsilon I\bmod{\Ical_N},\; \varepsilon\in\Z,\;\varepsilon^2\equiv 1\bmod{N}.
\]
Now start with an arbitrary $\widetilde{M}=(\widetilde{m}_{ij})\in G^{-1}_N \;\Dcal\bigl(U(N) \oplus U(N)\oplus (-T);\Z\bigr)\;G_N$. It follows from Theorem 3 in \cite{KRaW} that the preimage of $\widetilde{M}$ has the form
\[
 W_{\ell} L = L^{\ast} W_{\ell},\;\ell\mid d_\K\; \text{squarefree}, \;L,L^{\ast}\in \Gamma_2(\oh_\K),
\]
where $W_{\ell}$ is an arbitrary matrix in $(1/\sqrt{\ell})\, \Ical^{4\times 4}_{\ell}$ satisfying $\overline{W}^{tr}_{\ell} JW_{\ell} = J$ and $\det W_{\ell} = 1$. We choose
\begin{align*}
 W_{\ell} & = \begin{pmatrix}
               \overline{V}^{tr}_{\ell} & 0 \\ 0 & V^{-1}_{\ell}
              \end{pmatrix}, \;\; V_{\ell} = 
              \begin{pmatrix}
               \mu u/\sqrt{\ell} & \nu N\sqrt{\ell} \\ N\sqrt{\ell} & \overline{u}/\sqrt{\ell}
              \end{pmatrix},    \\[1ex]
 u & = \ell +m+\sqrt{-m}, \; \mu,\nu\in \Z,\;\; \mu u \overline{u}/\ell -\nu N^2 \ell = 1.
\end{align*}
Now let $m\equiv 1 \bmod 4$ and $N'= \gcd (N,m)$. As
\[
 \widetilde{W}_{\ell} = \begin{pmatrix}
                         1 & 0 & 0 \\ 0 & K_{\ell} & 0 \\ 0 & 0 & 1
                        \end{pmatrix}, \;\; K_{\ell} \equiv
                        \begin{pmatrix}
                         \mu & \ast & 0 \\ 0 & \ast & 0 \\ 0 & \ast & u\overline{u}/\ell
                        \end{pmatrix} \bmod{N},
\]
the standard procedure shows that $L^{\ast}$ is congruent to a diagonal matrix $\bmod{\;\Ical_N}$. The special choice of $W_{\ell}$ shows that this is also true for $L$, i.e.
\[
 L\equiv \diag (\varepsilon,\delta,\delta,\varepsilon)\bmod{\Ical_N}, \; \varepsilon,\delta\in\Z,\;\; \varepsilon \delta\equiv 1\bmod{N}.
\]
This leads to
\begin{align*}
 \widetilde{L} & \in \diag(1, \varepsilon^2,1,1,\delta^2,1)+\Lambda, \\
 \Lambda & = G^{-1}_N \Z^{6\times 6} \diag(N,N,2N',2mN/N',N,N).
\end{align*}
Thus we obtain for $\widetilde{W}_{\ell} = (\widetilde{w}_{ij})$
\begin{align*}
 \widetilde{w}_{43} & = -2\mu-2\mu m/\ell \equiv 0\bmod{2N'/N},  \\
 \widetilde{w}_{34} & \equiv 2\mu m(m+1) (m+\ell)/\ell \equiv 0 \bmod{2mN/N'},
\end{align*}
which implies $\ell\mid m$ by considering odd and even $N$ separately. Then
\[
 \widetilde{w}_{44} = 1-2\mu m(m+1)/\ell \equiv 1\bmod{2m/N'}
\]
leads to $\ell\mid N'$. Finally
\[
 \widetilde{w}_{33} \equiv 1+2 \mu m (m-1)/\ell \equiv 1 \bmod{2N'}
\]
gives $N'\ell\mid m$ and $\ell=1$, as $m$ is squarefree. Thus we get $L\in\Gamma_2(\oh_\K)$ and the claim follows from the considerations above.  \\
The (simpler) cases $m\equiv 2,3\bmod{4}$ are dealt with in the same way.
\end{proof}

\begin{remarkson}
 a) If $d_\K$ is even and $N\mid m$, the choice of another basis shows that 
 \[
 \Dcal\left( U(N) \oplus U(N) \oplus 
 \begin{pmatrix}
         -2N & 0 \\ 0 & -2m/N
        \end{pmatrix};\Z\right)
 \]       
 is isomorphic to the principal congruence subgroup $\bmod{\;\Ical_\N}$. \\
 b) If we conjugate by $\diag(\sqrt{-m},1,-1/\sqrt{-m},1)$ we see that the principal congruence subgroup $\bmod{\,\Ical_m}$ is isomorphic to the discriminant kernel $\Dcal\bigl(U(m) \oplus U(m) \oplus (-S_\K);\Z\bigr)$ via Theorem \ref{theorem_2}.
\end{remarkson}

\section{Quaternionic modular group}

In this section we consider the \emph{Hamiltonian quaternions}
\[
 \H = \R + \R i + \R j + \R k, \;\; k = ij = -ji, \;\; i^2 = j^2 = -1.
\]
If $b_1,$ $b_2,$ $b_3,$ $b_4$ is a basis of $\H$ over $\R$, we define 
\[
 S = (b_\nu \overline{b}_\mu + b_\mu \overline{b}_\nu)_{\nu,\mu}, \;\; S_0 = U(1)\oplus (-S),\;\; S_1 = U(1)\oplus S_0 \in\R^{8 \times 8}.
\]
We replace $\varphi$ in \eqref{gl_4n} by
\begin{gather*}\tag{4*}\label{gl_4stern_n}
 \varphi: \her_2(\H) \to \R^6,
\end{gather*}
which describes the coordinates with respect to the basis
\[
 \begin{pmatrix}
  1 & 0 \\ 0 & 0
 \end{pmatrix},\;\;
  \begin{pmatrix}
  0 & b_1 \\ \overline{b}_1 & 0
 \end{pmatrix},\;\;
  \begin{pmatrix}
  0 & b_2 \\ \overline{b}_2 & 0
 \end{pmatrix},\;\;
   \begin{pmatrix}
  0 & b_3 \\ \overline{b}_3 & 0
 \end{pmatrix},\;\;
  \begin{pmatrix}
  0 & b_4 \\ \overline{b}_4 & 0
 \end{pmatrix},\;\;
 \begin{pmatrix}
  0 & 0 \\ 0 & 1
 \end{pmatrix}. 
\]
If we extend $\varphi$ by $\C$-linearity, we obtain a bijection between the quaternionic half-space $H(2;\H)$ (cf. \cite{K3}, p. 46) and the orthogonal half-space $\Hcal_S$ (cf. \cite{G4}, \cite{K1}). We compare the action of 
\[
 Sp_2(\H):= \{M\in \H^{4\times 4};\; \overline{M}^{tr} JM = J\}
\]
on $H(2;\H)$ via $Z\mapsto M\langle Z\rangle = (AZ+B)(CZ+D)^{-1}$ with the action of $SO_0(S_1;\R)$ on $\Hcal_S$ (cf. \cite{G4})
\begin{align*}
 z\mapsto \widetilde{M}\langle z\rangle & : = \bigl(-\tfrac{1}{2} z^{tr} S_0 z\cdot b + Kz + c\bigr)\cdot (\widetilde{M}\{z\})^{-1},  \\
		     \widetilde{M}\{z\} & : = -\tfrac{1}{2} z^{tr} S_0 z\cdot \gamma + d^{tr} S_0 z + \delta,
\end{align*}
in the notation of $\widetilde{M}\in\R^{8\times 8}$ in \eqref{gl_5n} adapted to $S_0$ above.

\begin{lemma}\label{lemma_1} 
 $\varphi$ induces an isomorphism of the groups
\begin{gather*}\tag{20}\label{gl_20}
 Sp_2(\H)\big/\{\pm I\} \to SO_0(S_1;\R)\big/\{\pm I\}, \;\; \pm M\mapsto \pm \widetilde{M},
\end{gather*}
via 
\[
 \varphi(M\langle Z\rangle) = \widetilde{M}\langle \varphi(Z)\rangle \;\;\text{for all}\;\; Z\in H(2;\H).
\]
\end{lemma}

\begin{proof}
The assertion is verified for standard generators on both sides (cf. \cite{K3} Lemma II.1.4, \cite{K1}).
\end{proof}

Let ``$\,^\vee\,$'' stand for the standard embedding $\H\hookrightarrow \C^{2\times 2}$, i.e.
\[
 a = a_1 + a_2 i + a_3 j + a_4 k \mapsto \check{a} = \begin{pmatrix}
                                                  a_1 + a_2 i & a_3 + a_4 i \\ -a_3 + a_4 i & a_1-a_2 i
                                                 \end{pmatrix},
\]
and its extension to matrices (cf. \cite{K3} chap. I, \S2). Then a verification for the standard generators yields
\begin{gather*}\tag{21}\label{gl_21}
 \bigl(\widetilde{M}\{\varphi(Z)\}\bigr)^2 = \det(CZ+D)^\vee.
\end{gather*}
Note that the notation of $X^{\sharp}$ in \eqref{gl_3n} over the quaternions only makes sense for $X= \overline{X}^{tr}$. In this case we define
\[
 \det X = x_{11}x_{22}-x_{12}\overline{x}_{12} \;\;\text{with}\;\; (\det X)^2 = \det X^\vee.
\]
Now a verification leads to 
\begin{gather*}\tag{6*}\label{gl_6stern}
 \alpha = \pm\sqrt{\det A^\vee},\;\; \beta = \pm\sqrt{\det B^\vee},\;\; \gamma = \pm\sqrt{\det C^\vee}, \;\;\delta = \pm\sqrt{\det D^\vee},
\end{gather*}
\begin{gather*}\tag{7*}\label{gl_7stern}
 a = \pm\varphi(\alpha A^{-1} B),\;\; b = \pm\varphi(\gamma AC^{-1}),\;\; c = \pm\varphi(\delta BD^{-1}), \;\; d = \pm\varphi(\gamma C^{-1} D),
\end{gather*}
if the matrices involved have got rank $2$. 
Note that for instance also
\[
 b=\pm\varphi\bigl((\alpha CA^{-1})^\sharp\bigr),\;\; d = \pm \varphi \bigl((\delta D^{-1}C)^\sharp\bigr)
\]
hold. In the sequel $A$ and $D$ will always have rank $2$ in order to avoid technical difficulties. 
$K$ is again the representative matrix of an endomorphism $f_M$ with respect to the basis above. If $\det D^\vee\neq 0$ we have 
\begin{gather*}\tag{8*}\label{gl_8stern}
 f_M(Z) = \delta AZD^{-1} + BZ^{\sharp}(D^{-1} C)^{\sharp} \delta D^{-1}.
\end{gather*}
Note that $\delta \neq 0$ determines all the signs in \eqref{gl_6stern}, \eqref{gl_7stern} and \eqref{gl_8stern} uniquely in view of \eqref{gl_21}.

Now we consider the \emph{Hurwitz quaternions}
\[
 \Ocal = \Z+\Z i+\Z j+\Z\omega, \;\; \omega = \tfrac{1}{2}(1+i+j+k), \;\; S=
 \begin{pmatrix}
  2 & 0 & 0 & 1  \\ 0 & 2 & 0 & 1  \\  0 & 0 & 2 & 1  \\  1 & 1 & 1 & 2
 \end{pmatrix}.
\]
If $0\neq X\in \Ocal^{2\times 2}$ let
\[
 \rho(X):=\max\{\ell \in\N;\;\tfrac{1}{\ell} X\in\Ocal^{2\times 2}\}.
\]
Moreover we consider the prime ideal $\wp$ of even quaternions
\[
 \wp = \{a\in\Ocal;\;a\overline{a}\in 2\Z\} = \Z 2 + \Z(1+i) + \Z(1+j) + \Z(1+k).
\]
Note that 
\[
 \Ocal/\wp = \{\wp,1+\wp,\omega+\wp,\overline{\omega}+\wp\}\cong \F_4.
\]

\begin{lemma}\label{lemma_2}   
 If $X$ is any block among $A,B,C,D$ of $M\in Sp_2(\Ocal)$ with $\det X^\vee\neq 0$, then there exist $U,V\in GL_2(\Ocal)$ as well as $m,n\in\N$ such that
 \[
  UXV = \begin{pmatrix}
         m & 0 \\ 0 & mn
        \end{pmatrix} \;\;\text{or}\;\;\begin{pmatrix}
				      m(1+i) & 0 \\ 0 & mn(1+i)
				      \end{pmatrix},
 \]
where $n$ is odd in the latter case. Moreover one has 
\begin{enumerate}
 \item[a)] $\sqrt{\det X^\vee}\in\N$.
 \item[b)] $\sqrt{\det X^\vee} X^{-1} \in\Ocal^{2\times 2}$.
 \item[c)] $m=\rho(X) = \rho\bigl(\sqrt{\det X^\vee} X^{-1}\bigr)$.
\end{enumerate}
\end{lemma}

\begin{proof}
Let $X=A$ be of rank $2$ without restriction. Now choose $U,V\in GL_2(\Ocal)$ such that $A^\ast = UAV$ is in elementary divisor form (cf. \cite{K7}), i.e.
\begin{gather*}
  A^\ast = \begin{pmatrix}
            a_1 & 0 \\ 0 & a_4
           \end{pmatrix}, \;\;
 \begin{split}
  & a_\nu = u_\nu \alpha_\nu , \; u_\nu = n_\nu\;\text{or}\; n_\nu(1+i),\; n_\nu \in\N,   \\
  & \alpha_\nu \overline{\alpha}_\nu \;\text{odd}, \; \tfrac{1}{m_\nu} \alpha_\nu \notin \Ocal,\; m_\nu\in\N, \; m_\nu>1,\; a_1\overline{\alpha}_\nu\mid u_4.
 \end{split}
\end{gather*}
As $A^*\overline{B^*}^{tr}$, $B^* = UB\overline{V}^{tr-1} = \left(\begin{smallmatrix}
                                                                   b_1 & b_2 \\ b_3 & b_4
                                                                  \end{smallmatrix}\right)$,
is Hermitian, we obtain
\begin{align*}
 b_\nu & = m_\nu \alpha_\nu \;\,\text{resp.}\;\, b_\nu = m_\nu(1+i)\alpha_\nu,\; m_\nu\in \N, \;\,\text{if} \;\, \nu=1,4, \\
 b_3 & = a_4\overline{b}_2 a_1^{-1}, \; b_2\in\Ocal.
\end{align*}
If $P$ is any $2\times 2$ submatrix of $(A^*, B^*)$, we see that $\det P^\vee$ is divisible by $\alpha_1 \overline{\alpha}_1 \cdot \alpha_4\overline{\alpha}_4$ in any case as well as by $2$ if
\[
 u_1^{-1} u_4 \notin \Z \;\;\text{or}\;\; u_\nu = m_\nu(1+i),\;\nu=1,4 \;\;\text{and}\;\; u_1^{-1} u_4 \in 2\Z.
\]
As $(A^*,B^*)$ is the upper triangular $2\times 4$ block of a matrix in $GL_4(\Ocal)$, the determinants $\det P^\vee$ are coprime, whenever $P$ runs through all $2\times 2$ blocks of $(A^*,B^*)$. This leads to the result in view of
\[
 \sqrt{\det A^\vee} V^{-1} A^{-1} U^{-1} = 
 \begin{pmatrix}
  mn & 0  \\  0 & m
 \end{pmatrix} \;\;\text{or}\;\;\begin{pmatrix}
				mn(1-i) & 0  \\  0 & m(1-i)
				\end{pmatrix}.
\]
\vspace*{-8ex}\\
\end{proof}
\vspace{4ex}
Next we want to describe the discriminant kernel. Therefore let
\[
 \Gamma_2(\Ocal):= \{M\in Sp_2(\Ocal);\;\det (M \bmod{\wp}) \equiv 1\bmod{\wp}\}
\]
stand for the special modular group, which satisfies
\[
 Sp_2(\Ocal) = \Gamma_2(\Ocal) \cup (\omega I)\cdot \Gamma_2(\Ocal) \cup (\overline{\omega} I)\cdot \Gamma_2(\Ocal),
\]
as well as (cf. \cite{FH})
\[
 \Gamma^*_2 (\Ocal) = Sp_2(\Ocal) \cup \left(\frac{1+i}{\sqrt{2}} I\right) Sp_2(\Ocal).
\]

\begin{theorem}\label{theorem_4}  
 a) In \eqref{gl_20} the special modular group $\Gamma_2(\Ocal)\big/\{\pm I\}$ is isomorphic to the discriminant kernel $\Dcal(S_1;\Z)\big/\{\pm I\}$.  \\
 b) In \eqref{gl_20} the extended modular group $\Gamma^*_2(\Ocal)\big/\{\pm I\}$ is isomorphic to $SO_0(S_1;\Z)\big/\{\pm I\}$.
\end{theorem}

\begin{proof}
 Verify the inclusions for generators of the groups above (cf. \cite{K3} Theorem II.2.3, \cite{FH}, \cite{K1}). Note that $SO(S;\Z)$ is a group of order $576$, which is described in \cite{FH}, and use
 \begin{gather*}\tag{22}\label{gl_22}
  (\Z\times \Z\times \Z\times \Z)S = (2\Z\times 2\Z\times 2\Z\times \Z)\cup (1,1,1,0) + (2\Z\times 2\Z\times 2\Z\times \Z). 
 \end{gather*}
 \vspace*{-8ex}\\
\end{proof}

Now we are going to describe congruence subgroups of $Sp_2(\Ocal)$. Note that the congruence $X\equiv I\bmod{N}$ in Lemma \ref{lemma_2} does \textbf{not} imply $\sqrt{\det X^\vee}\equiv 1 \bmod{N}$ in general. As an example consider
\[
 N= 15, \;\; X=\diag(1+15(20i+76j+280k),1+15\omega),
\]
\[
 \sqrt{\det X^\vee} = 67,721\equiv 11\bmod{15}.
\]
Therefore we use a variant of the definition of the discriminant kernel.

\begin{corollary}\label{corollary_1} 
 Given $N\in\N$, then $\eqref{gl_20}$ yields an isomorphism between 
\[
 \{M\in \Gamma_2(\Ocal);\; M\equiv \varepsilon I\bmod{N},\,\varepsilon\in\Z,\,\varepsilon^2\equiv 1 \bmod{N}\}\big/\{\pm I\},
\]
if $N$ is odd, resp.
\[
 \{M\in \Gamma_2(\Ocal); M\equiv \varepsilon I\bmod{N},\varepsilon\in\Z,\varepsilon^2\equiv 1 \bmod{N}, \varepsilon(a_{11}+ a_{22}) \equiv \varepsilon^2 +1\bmod{N\wp} \}\big/\{\pm I\}, 
\]
if $N$ is even, and 
 \begin{gather*}\tag{23}\label{gl_23}
 \{\widetilde{M}\in SO_0(S_1;\Z);\; \widetilde{M}\in\rho I + \Z^{8\times 8} NS_1,\,\rho\in\Z\;\text{odd},\,\rho^2\equiv 1 \bmod{N}\}\big/\{\pm I\}.
 \end{gather*}
\end{corollary}

\begin{proof}
 Note that \eqref{gl_23} is contained in $\Dcal(S_1;\Z)/\{\pm I\}$. Proceed in the same way as in the proof of Theorem \ref{theorem_2}. Use \eqref{gl_6stern}, \eqref{gl_7stern}, \eqref{gl_8stern} as well as $\alpha \delta \equiv 1 \bmod{N}$, whenever $\widetilde{M}$ is congruent to a diagonal matrix $\bmod{\;N}$. 
 If $M$ is congruent to a diagonal matrix $\bmod{\;N}$, one observes the conditions
 \begin{align*}
  a_{11} u\overline{a}_{22} & \equiv u\bmod{N} \;\text{for all}\; u\in\Ocal \\
  \text{and for $N$ even } a_{11} \overline{a}_{22}-1 & \equiv a_{11} i\overline{a}_{22} - i \equiv a_{11} j\overline{a}_{22} - j \bmod{2N},
 \end{align*}
due to \eqref{gl_22}, which leads to the claim.
\end{proof}

The additional condition for even $N$ above in particular ensures that $\sqrt{\det A^{\vee}}\equiv \pm 1 \bmod{2^k}$, if $N=2^k N'$ with odd $N'$.

Finally we consider the congruence $\bmod{\wp}$.

\begin{corollary}\label{corollary_2} 
 The principal congruence subgroup
 \[
  \{M\in Sp_2(\Ocal);\; M\equiv \varepsilon I \bmod{\wp},\, \varepsilon=1,\omega,\overline{\omega}\}/\{\pm I\}
 \]
is isomorphic to 
\[
 \Dcal\bigl(U(2)\oplus U(2) \oplus (-S);\Z\bigr)/\{\pm I\}.
\]
\end{corollary}

\begin{proof}
 Proceed in the same way as in the proof of Theorem \ref{theorem_2}. Matrices $M$ in the preimage of $F\;\Dcal\bigl(U(2) \oplus U(2) \oplus (-S);\Z\bigr)\;F^{-1}$ satisfy
\[                                                                                                                                                                       a_{21}\equiv b_{22} \equiv d_{12} \equiv c_{12} \equiv c_{22} \equiv0\bmod{2}, \;\; c_{11}\equiv 0 \bmod{4}
\]                                                                                                                                                              as well as $\det (M \bmod{\wp})\equiv 1\bmod{\wp}$. Hence the diagonal entries of $M$ are odd quaternions. Thus the diagonal is congruent to 
\[
\diag(\varepsilon,\overline{\varepsilon},\varepsilon, \overline{\varepsilon}) \bmod{\wp}, \; \varepsilon = 1, \omega,\overline{\omega}.
\]
Now conjugate by $\diag (1+i,1,\frac{1}{2}(1+i),1)$. In view of 
\[
(1+i) \omega (1+i)^{-1} \in \overline{\omega} + \wp
\]
the claim follows. 
\end{proof}

\begin{remarkson}
 a) All the elementary divisor forms in Lemma \ref{lemma_2} actually occur:
\[
  \begin{pmatrix}
   X & -I  \\  I & 0
  \end{pmatrix}, \;\; X=\diag(m,mn) \;\;\text{or}\;\; X= m
\begin{pmatrix}
 2 & 1+i  \\ 1-i & n+1
\end{pmatrix}.
\]
 b) Clearly one can also consider the other types of discriminant kernels of the form $\Dcal\bigl(U(N)\oplus U(n)\oplus (-S);\Z\bigr)$ just as in Theorem \ref{theorem_2}.  \\
 c) The case $SO_0(2,10)$ was dealt with in the same way as in Theorem \ref{theorem_4} in \cite{DKW}.
\end{remarkson}
\vspace{2ex}
\noindent
\textbf{Acknowledgement.} The authors thank Brandon Williams for suggesting the topic and the referee for helpful comments.

\vspace{6ex}


\bibliography{bibliography_krieg_2020} 
\bibliographystyle{plain}

\end{document}